\documentclass[10pt]{article}

\usepackage[T1]{fontenc}
\usepackage{lmodern}
\usepackage[a4paper,margin=2.35cm]{geometry}
\usepackage{amsmath,amssymb,amsthm,mathtools}
\usepackage{microtype}
\usepackage[hidelinks]{hyperref}

\AtBeginDocument{%
  \setlength{\abovedisplayskip}{8pt plus 2pt minus 3pt}%
  \setlength{\belowdisplayskip}{8pt plus 2pt minus 3pt}%
  \setlength{\abovedisplayshortskip}{0pt plus 2pt}%
  \setlength{\belowdisplayshortskip}{5pt plus 2pt minus 2pt}%
}

\newtheorem{theorem}{Theorem}[section]
\newtheorem{proposition}[theorem]{Proposition}
\newtheorem{lemma}[theorem]{Lemma}
\newtheorem{corollary}[theorem]{Corollary}
\theoremstyle{definition}
\newtheorem{definition}[theorem]{Definition}

\newcommand{\F}{\mathbb F}
\newcommand{\wt}{\operatorname{wt}_{\rm rk}}
\newcommand{\GL}{\operatorname{GL}}
\newcommand{\Rdef}{\operatorname{Rdef}}
\newcommand{\qbinom}[2]{\genfrac{[}{]}{0pt}{}{#1}{#2}_{q}}
\title{The extremal gap for scattered $q$-systems}
\author{Alessandro Giannoni\thanks{Dipartimento di Matematica e Applicazioni
``R. Caccioppoli'', Universit\`a degli Studi di Napoli Federico II,
Napoli, Italy. Email: \texttt{alessandro.giannoni@unina.it}.}}
\date{}

\hypersetup{
  pdftitle={The extremal gap for scattered q-systems},
  pdfauthor={Alessandro Giannoni}
}

\begin{document}

\maketitle

\begin{abstract}
Two general estimates govern the rank of an $h$-scattered
$q$-system: the scattered upper bound and the lower bound for
systems maximal under inclusion. We determine exactly when they
coincide. Besides $m=h+1$ and $m=h+2$, equality occurs precisely
when $h$ is even, $m=h+3$, and
$k\equiv h/2\pmod{h+1}$. We also prove rigidity in the first two
cases: every extremal system is equivalent to either $\F_q^k$ or a
direct sum of elementary Gabidulin systems and $\F_q$-directions.

For $h=m-3$ and $k=\ell(m-2)+s$, with
$1\leq s<(m-2)/2$, the two bounds differ by one, apart from one
boundary case. Starting from an arbitrary maximum scattered
$q$-system, we construct extremal quasi-maximum $(m-3)$-scattered
systems of rank $\ell m+s$. Every such extension preserves lower
bounds on the generalized rank weights of the initial code. When an
explicit independence condition holds, the construction also
produces a nondegenerate dual pair of quasi-MRD codes. We parametrize
these extensions by
subspaces of a quotient space and derive a uniform lower bound on the
number of inequivalent outputs. Known order-two families yield codes
separated from every nontrivial direct sum under an explicit
numerical condition, while higher-order families are separated from
a specified order-one direct-sum class. We conclude with a fully
explicit binary example.
\end{abstract}

\medskip
\noindent\textbf{Keywords.}
$q$-systems; $h$-scattered subspaces; graph extensions;
quasi-maximum scattered subspaces; quasi-MRD codes; generalized rank
weights.

\smallskip
\noindent\textbf{2020 Mathematics Subject Classification.}
51E20; 94B05; 94B27; 11T06.

\section{Introduction}

Let $F=\F_{q^m}$.  A rank-$n$ $q$-system in $F^k$ is an
$n$-dimensional $\F_q$-subspace spanning $F^k$ over $F$.  It is
$h$-scattered if
$$
 \dim_{\F_q}(U\cap H)\leq h
$$
for every $h$-dimensional $F$-subspace $H$.  Its rank satisfies
\begin{equation}\label{eq:intro-upper}
 \dim_{\F_q}(U)\leq
 \left\lfloor\frac{km}{h+1}\right\rfloor.
\end{equation}
Systems attaining \eqref{eq:intro-upper} are \emph{extremal}, and
\emph{quasi-maximum} when $h+1\nmid km$; the upper bound is
\cite[Theorem~2.3]{CMPZ}, and also follows from
\cite[Corollary~4.4]{MNT}.  Explicit quasi-maximum systems with
$h>1$ remain comparatively scarce.

The word \emph{maximally} has a different meaning: it refers to
maximality under inclusion.  Every maximally $h$-scattered system
satisfies the general lower bound
$$
 \dim_{\F_q}U\geq
 \left\lceil\frac{m(k-h)}{h+1}\right\rceil+h.
$$
This is \cite[Theorem~2.4]{BGMN}.
Thus there is an a priori interval between the rank guaranteed by
maximality and the extremal rank in \eqref{eq:intro-upper}.  A natural
first question is when this interval collapses.

Theorem~\ref{thm:bounds-equality} answers this question completely.
Besides the two uniform cases $m=h+1$ and $m=h+2$, the bounds
coincide only in the last near-diagonal case
$$
 h\text{ even},\qquad m=h+3,\qquad
 k\equiv\frac h2\pmod{h+1}.
$$
The two uniform cases might appear to require additional
constructions.  Proposition~\ref{prop:rigid-equality-cases} shows
instead that they are completely rigid.  If $m=h+1$, every extremal
system is equivalent to $\F_q^k$.  If $m=h+2$, every extremal system
is equivalent to a direct sum of elementary Gabidulin systems and
$\F_q$-directions.  Thus the familiar stabilization constructions
exhaust both cases up to equivalence.

The case $m=h+3$ is consequently the first one in which the equality
of the bounds may coexist with genuinely new equivalence types.  It
is also the last such regime, since no equality is possible when
$m\geq h+4$.  This singles out $h=m-3$ as the natural next case.

We work with the specific value $h=m-3$ and write
$$
 h=m-3,\qquad
 k=\ell(m-2)+s,
 \qquad n=\ell m+s.
$$
Corollary~\ref{cor:specialized-extremal-gap} specializes
Theorem~\ref{thm:bounds-equality} to these parameters.  If
$1\leq s<(m-2)/2$, the upper bound is $n$, whereas the general
maximality bound is
$$
 n-1,
$$
unless $m$ is odd and $s=(m-3)/2$, in which case it is already $n$.
Consequently, away from this boundary case, the general theory leaves
a one-dimensional extremal gap.  The motivating problem of the paper
is to attain the upper bound throughout this range by explicit
constructions, while retaining nondegenerate coding-theoretic
information.  In the boundary case, the equality of the bounds
explains extremality once maximality is known, but does not itself
provide explicit systems or the additional properties studied here.

Let $W\leq F^{2\ell}$ be maximum scattered of rank $\ell m$; we call
$\ell$ the \emph{order} of $W$.  Choose a basis matrix $H_0$ of
$W$.  The coding operation used below is the transparent column
extension
$$
 H_0\longmapsto [\,H_0\mid Z\,].
$$
If $U_0\leq F^{\ell(m-2)}$ is the dual system, the dual construction
has the graph form
\begin{equation}\label{eq:intro-graph}
 U_\varphi=
 \{(u,\varphi(u)+a):u\in U_0,\ a\in\F_q^s\}.
\end{equation}
A projection argument shows that every $U_\varphi$ is
$(m-3)$-scattered; in the range above it is extremal and
quasi-maximum. The choice $\varphi=0$ gives the elementary extension obtained by
adjoining $\F_q$-directions that also appears in cone constructions
\cite{AMSZ}, but its dual code is degenerate.  Puncturing onto the
first $\ell m$ coordinates shows, for every $\varphi$ and without
any additional hypothesis, that no generalized rank weight of the
extended dual code decreases.  When the images modulo $W$ of the
columns of $Z$ are $\F_q$-linearly independent, both dual codes are
nondegenerate and quasi-MRD, so the output cannot be equivalent to
the elementary extension.  For $\ell=1$, the extensions satisfying
this condition are contained in \cite[Theorem~4.4]{BBMS}.

The general geometric interpretation of Delsarte duality and the
duality properties of quasi-MRD codes belong to the framework of
\cite{BPZ}.  The projection argument behind graph stabilization is
elementary and is not, by itself, the main novelty.  The new content
lies in the equality classification and rigidity results above and,
within the extension framework, in the explicit dual formulation in
arbitrary order, the characterization of nondegeneracy,
generalized-weight transfer, the uniform orbit estimate, and the
resulting decomposition-separation statements.

The parameters do not merely describe many choices of matrices.  The
extensions of a fixed $W$ satisfying the independence condition are
parametrized by the $s$-subspaces of the quotient $F^{2\ell}/W$.
Dividing their number by the order of the full semilinear group gives
a uniform lower bound without using stabilizer information.  With
$s=\lfloor(m-3)/2\rfloor$, this lower bound tends to infinity for
fixed $q$ and $\ell$ as $m$ grows.

To separate the new systems, distribute $s=s_1+\cdots+s_\ell$ and
take a direct sum of nondegenerate
$[m+s_i,2,m-1]_{q^m/q}$ codes.  Its second generalized rank weight is
$$
 m+\min_i s_i\leq m+\left\lfloor\frac{s}{\ell}\right\rfloor.
$$
Hence a larger second weight excludes the entire comparison class.
For the order-two systems of \cite{BMNV}, an explicit numerical
condition excludes every nontrivial direct sum.  The higher-order
systems of \cite{BGM} give examples for every $\ell\geq3$ that are
separated from this order-one comparison class; we do not claim full
indecomposability in higher order.  This restricted comparison is
different from the completely decomposable codes of
\cite{Santonastaso}, whose factors are one-dimensional MRD codes.  We
also give a binary example with exact second weight.

\section{The extremal gap and preliminaries}\label{sec:preliminaries}

Throughout, $q=p^e$ is a prime power, $F=\F_{q^m}$, and dimensions
without an indicated ground field are taken over $F$.  By
$\operatorname{Aut}(F)$ we mean the full field automorphism group,
of order $em$; every such automorphism stabilizes $\F_q$ setwise.
This deliberately broader convention, rather than restricting to
$\operatorname{Gal}(F/\F_q)$, is used throughout.
For a matrix $M$ over $K$, write $\operatorname{rowsp}_K(M)$ for its
row space and $M^{-\top}=(M^{-1})^\top$ when $M$ is invertible.

\begin{definition}
An $[n,k]_{q^m/q}$ \emph{$q$-system} is an $n$-dimensional
$\F_q$-subspace $U\leq F^k$ with $\langle U\rangle_F=F^k$.
It is \emph{$h$-scattered} if
$\dim_{\F_q}(U\cap H)\leq h$ for every $h$-dimensional
$F$-subspace $H\leq F^k$.  It is \emph{extremal} if
$n=\lfloor km/(h+1)\rfloor$.  An extremal $h$-scattered system is
\emph{maximum} if $h+1\mid km$, and \emph{quasi-maximum} otherwise.
A maximum scattered system in $F^{2\ell}$ means a $1$-scattered
system of rank $\ell m$.

Two systems in $F^k$ are \emph{equivalent} if
$U'=B(\sigma(U))$ for some $B\in\GL(k,F)$ and
$\sigma\in\operatorname{Aut}(F)$.
\end{definition}

An $h$-scattered system is \emph{maximally $h$-scattered} if it is
maximal with respect to inclusion among the $h$-scattered
$\F_q$-subspaces of $F^k$.

\begin{theorem}\label{thm:bounds-equality}
Assume $h\geq1$, $m\geq h+1$, and $k\geq h+1$.  Every
$h$-scattered $U\leq F^k$ satisfies
$$
 \dim_{\F_q}U\leq
 M(k,m,h):=\left\lfloor\frac{km}{h+1}\right\rfloor,
$$
whereas every maximally $h$-scattered $U$ satisfies
$$
 \dim_{\F_q}U\geq
 L(k,m,h):=
 \left\lceil\frac{m(k-h)}{h+1}\right\rceil+h.
$$
Moreover, $L(k,m,h)=M(k,m,h)$ if and only if
$$
 m=h+1,\qquad m=h+2,
$$
or
$$
 h\text{ is even},\qquad m=h+3,\qquad
 k\equiv\frac h2\pmod{h+1}.
$$
\end{theorem}

\begin{proof}
The upper bound is \cite[Theorem~2.3]{CMPZ}; see also
\cite[Corollary~4.4]{MNT}.  The lower bound for inclusion-maximal
systems is \cite[Theorem~2.4]{BGMN}.  We determine when these two
integer bounds coincide.  Put $N=h+1$ and write
$m=aN+r$ and $k=cN+u$, where $a,c\geq1$ and
$0\leq r,u<N$.  Direct substitution shows that $L=M$ is equivalent
to
$$
 (1-a)h-r+
 \left\lceil\frac{r(u+1)}N\right\rceil
 -\left\lfloor\frac{ru}N\right\rfloor=0.
$$
The difference of the last two terms is zero if $r=0$ and belongs to
$\{1,2\}$ otherwise.  If $a\geq2$, the left-hand side is negative
for $r=0$ and is at most $-h-r+2$ for $r>0$; equality in this estimate
could occur only for $h=r=1$, when the actual difference of the last
two terms is one and the left-hand side is $-1$.  Hence $a=1$, and
the equality reduces to
$$
 \left\lceil\frac{r(u+1)}N\right\rceil
 -\left\lfloor\frac{ru}N\right\rfloor=r.
$$
Thus $r=0$ or $r=1$, with arbitrary $u$, or $r=2$ and
$2u<N<2u+2$.  The last condition is equivalent to
$N=2u+1$.  Translating $r$ and $u$ back into $m$ and $k$ gives the
three cases above.
\end{proof}

\begin{corollary}\label{cor:specialized-extremal-gap}
Let $m\geq5$, $\ell\geq1$, $h=m-3$, and
$k=\ell(m-2)+s$, where $1\leq s<(m-2)/2$.  Then
$$
 M(k,m,h)=\ell m+s
$$
and
$$
 L(k,m,h)=
 \begin{cases}
  \ell m+s,&m\text{ is odd and }s=(m-3)/2,\\
  \ell m+s-1,&\text{otherwise}.
 \end{cases}
$$
Thus, except in the first case, the general lower bound for
maximally $(m-3)$-scattered systems lies exactly one below the
extremal rank.
\end{corollary}

\begin{proof}
Since $h+1=m-2$ and $2s<m-2$,
$$
 M(k,m,h)
 =\ell m+s+\left\lfloor\frac{2s}{m-2}\right\rfloor
 =\ell m+s.
$$
Moreover,
$$
 L(k,m,h)
 =\ell m+s-2+
 \left\lceil\frac{2(s+1)}{m-2}\right\rceil.
$$
The last ceiling is two precisely when $m$ is odd and
$s=(m-3)/2$; it is one in every other case allowed by the
hypotheses.
\end{proof}

\begin{definition}
An $[n,r]_{q^m/q}$ \emph{rank-metric code} is an $r$-dimensional
$F$-subspace $C\leq F^n$.  For $x=(x_1,\ldots,x_n)$, its rank
weight is
$\wt(x)=\dim_{\F_q}\langle x_1,\ldots,x_n\rangle_{\F_q}$, and
$d(C)=\min\{\wt(x):x\in C\setminus\{0\}\}$.  Its dual is
$C^\perp=\{y\in F^n:xy^\top=0\text{ for every }x\in C\}$.
\end{definition}

Fix an $\F_q$-basis of $F$ and expand $x\in F^n$ as an
$m\times n$ matrix over $\F_q$; its row space is
$\operatorname{Rsupp}_q(x)$.  For an $F$-subspace $V\leq F^n$, let
$\operatorname{Rsupp}_q(V)$ be the sum of the supports of a basis of
$V$.  These spaces do not depend on the chosen bases.  The $i$-th
generalized rank weight is
\begin{equation}\label{eq:generalized-rank-weight}
 d_i(C):=\min\left\{
 \dim_{\F_q}\operatorname{Rsupp}_q(V):
 V\leq C,\ \dim_F V=i
 \right\}.
\end{equation}

Two $F$-linear codes $C,C'\leq F^n$ are \emph{rank-metric
equivalent} if $C'=\sigma(C)A$ for some $A\in\GL(n,q)$ and
$\sigma\in\operatorname{Aut}(F)$.  The Singleton bound is
\begin{equation}\label{eq:singleton}
 mr\leq
 \min\{m(n-d(C)+1),\,n(m-d(C)+1)\}.
\end{equation}
A code meeting this bound is MRD.  For $n>m$, its rank defect is
$$
 \Rdef(C)=m-\left\lceil\frac{mr}{n}\right\rceil-d(C)+1.
$$
A code with rank defect zero is \emph{quasi-MRD}
\cite{Delsarte,DGLR,Gabidulin,MNT}.

Let $H\in F^{r\times n}$ be a generator matrix of an
$[n,r]_{q^m/q}$ code $D$.  It is \emph{nondegenerate} if the columns
of $H$ are $\F_q$-independent.  Their $\F_q$-span is then an
$[n,r]_{q^m/q}$ $q$-system, and every $q$-system arises in this way.
This correspondence preserves equivalence classes
\cite[Theorem~2.15]{MNT}.

\begin{lemma}\label{lem:equivalence-duality}
If $C'=\sigma(C)A$, with $A\in\GL(n,q)$, then
$$
 (C')^\perp=\sigma(C^\perp)A^{-\top},
 \qquad d_i(C')=d_i(C)\qquad(1\leq i\leq\dim_F C).
$$
The code--system correspondence and equivalence remain compatible
after taking duals.
\end{lemma}

\begin{proof}
The operations $\sigma$ and $A$ preserve rank-support dimensions,
and the duality formula follows directly from the inner product.
The remaining assertion follows by applying the same operations to
a generator matrix.
\end{proof}

Suppose $D$ is nondegenerate and $d(D)>1$, and let $G$ be a generator
matrix of $D^\perp$.  Its columns are $\F_q$-independent, since an
$\F_q$-relation would be a rank-one word of $D$.  Their span is the
\emph{dual $q$-system} of the column system of $D$; it is well defined
up to equivalence by the code--system correspondence
\cite[Theorem~2.15]{MNT}.

\begin{theorem}\label{thm:bridge}
Let $C$ be a nondegenerate $[n,k]_{q^m/q}$ code, let $U$ be an
associated $q$-system, and assume $h<k<n$.  Then
$$
 U\text{ is }h\text{-scattered}
 \quad\Longleftrightarrow\quad
 d(C^\perp)\geq h+2.
$$
\end{theorem}

This is \cite[Theorem~3.3]{MNT}.  We also use that a nondegenerate
$[n,r]_{q^m/q}$ code satisfies $d_r=n$ and has strictly increasing
generalized rank weights \cite{Ducoat,KMU,MNT}.

\begin{proposition}\label{prop:rigid-equality-cases}
Let $U\leq F^k$ be an extremal $h$-scattered $q$-system, where
$k\geq h+1$.

If $m=h+1$, then $U$ is equivalent to $\F_q^k$.

If $m=h+2$, write
$$
 k=a(m-1)+r,\qquad a\geq1,\qquad 0\leq r<m-1,
$$
and put
$$
 \mathcal G_m=
 \{(x,x^q,\ldots,x^{q^{m-2}}):x\in F\}\leq F^{m-1}.
$$
Then $U$ is equivalent to
$$
 \mathcal G_m^{\oplus a}\oplus\F_q^r\leq F^k,
$$
where the summands occupy disjoint $F$-coordinate blocks.  In
particular, the cases $m=h+1$ and $m=h+2$ yield no equivalence
classes beyond the elementary $\F_q$ and Gabidulin direct sums.
\end{proposition}

\begin{proof}
Assume first that $m=h+1$.  Extremality gives
$\dim_{\F_q}U=k$.  Any $\F_q$-basis of $U$ consists of $k$ vectors
that span $F^k$ over $F$, hence it is an $F$-basis of $F^k$.  An
element of $\GL(k,F)$ maps this basis to the standard basis, so
$U$ is equivalent to $\F_q^k$.

Now let $m=h+2$.  The extremal rank is
$$
 n=\left\lfloor\frac{km}{m-1}\right\rfloor=am+r.
$$
Let $C$ be the nondegenerate $[n,k]_{q^m/q}$ code associated with
$U$ and put $D=C^\perp$.  Then $\dim_FD=n-k=a$, and
Theorem~\ref{thm:bridge} gives
$$
 d(D)\geq h+2=m.
$$
Since rank weight is at most $m$, equality holds.

Let $H\in F^{a\times n}$ be a generator matrix of $D$, and let $S$
be the $\F_q$-span of its columns in $F^a$.  Suppose that
$S\neq F^a$.  The trace pairing
$$
 (x,y)\longmapsto
 \operatorname{Tr}_{F/\F_q}(xy^\top)
$$
on $F^a$ is nondegenerate, so there is a nonzero $x\in F^a$ such
that $\operatorname{Tr}_{F/\F_q}(xy^\top)=0$ for every $y\in S$.
The word $xH$ is nonzero because $H$ has full row rank, while all
its coordinates belong to the trace-zero hyperplane of $F$.
Consequently $\wt(xH)\leq m-1$, contradicting $d(D)=m$.  Hence
$S=F^a$.

Fix an $\F_q$-basis
$\boldsymbol\beta=(\beta_1,\ldots,\beta_m)$ of $F$.  Since the
columns of $H$ span the $am$-dimensional $\F_q$-space $F^a$, there
is an $A\in\GL(n,q)$ such that
$$
 HA=
 \begin{pmatrix}
  \boldsymbol\beta&0&\cdots&0&0\\
  0&\boldsymbol\beta&\cdots&0&0\\
  \vdots&\vdots&\ddots&\vdots&\vdots\\
  0&0&\cdots&\boldsymbol\beta&0
 \end{pmatrix},
$$
where the final block has width $r$.  Thus $D$ is rank-metric
equivalent to
$$
 D_{\boldsymbol\beta}^{\oplus a}\oplus0^r,
 \qquad
 D_{\boldsymbol\beta}=F\boldsymbol\beta,
$$
and $D_{\boldsymbol\beta}$ is a one-dimensional $[m,1,m]_{q^m/q}$
Gabidulin code.  After deleting the final zero block, this is a
completely decomposable code in the sense of
\cite[Theorem~3.2]{Santonastaso}.  By
Lemma~\ref{lem:equivalence-duality}, $C$ is equivalent to
$$
 (D_{\boldsymbol\beta}^{\perp})^{\oplus a}\oplus F^r.
$$
To identify the corresponding system, choose an $\F_q$-basis
$\boldsymbol\xi=(\xi_1,\ldots,\xi_m)$ of $F$ and form the Moore
matrix
$$
 M_{\boldsymbol\xi}
 =(\xi_j^{q^i})_{\substack{0\leq i\leq m-2\\1\leq j\leq m}}.
$$
Its row space is an $[m,m-1,2]_{q^m/q}$ Gabidulin code
\cite{Gabidulin}, so its dual is one-dimensional and has minimum
rank distance $m$.  The coordinates of a generator of that dual
therefore form an $\F_q$-basis of $F$.  An element of $\GL(m,q)$
sends this basis to $\boldsymbol\beta$, and hence
$\operatorname{rowsp}_F(M_{\boldsymbol\xi})$ is rank-metric
equivalent to $D_{\boldsymbol\beta}^{\perp}$.  The columns of
$M_{\boldsymbol\xi}$ span $\mathcal G_m$, so the code--system
correspondence gives the assertion.
\end{proof}

\section{Graph extensions and dual \texorpdfstring{$q$}{q}-systems}
\label{sec:stabilization}

Proposition~\ref{prop:rigid-equality-cases} shows that the first two
equality cases of Theorem~\ref{thm:bounds-equality} contain only the
elementary equivalence classes.  For the next case,
Corollary~\ref{cor:specialized-extremal-gap} reduces the extremal
problem to gaining, in general, one dimension beyond the rank
ensured by inclusion-maximality.  We now construct systems attaining
the upper bound uniformly.

Let $m\geq5$, $\ell\geq1$, and put
$$
 h=m-3,\qquad n_0=\ell m,\qquad k_0=\ell(m-2).
$$
Let $W\leq F^{2\ell}$ be a maximum scattered $q$-system of rank
$n_0$, choose an ordered $\F_q$-basis of $W$ as the columns of
$H_0$, and put $D_0=\operatorname{rowsp}_F(H_0)$.  By
\cite[Theorem~4.9]{MNT},
\begin{equation}\label{eq:initial-code-parameters}
 D_0\text{ has parameters }[n_0,2\ell,m-1]_{q^m/q}.
\end{equation}
Let $C_0=D_0^\perp$, choose a generator matrix
$G_0\in F^{k_0\times n_0}$, and let $U_0\leq F^{k_0}$ be its
column system.  Since $d(D_0)>1$, the columns of $G_0$ are
$\F_q$-independent; the row rank of $G_0$ shows that they span
$F^{k_0}$ over $F$.  Thus $U_0$ is an $[n_0,k_0]_{q^m/q}$
$q$-system.  Since $d(D_0)=m-1=h+2$, Theorem~\ref{thm:bridge}
shows that $U_0$ is $h$-scattered.  Its rank is $n_0$, which equals
$k_0m/(h+1)$, so $U_0$ is maximum $h$-scattered.

For an $\F_q$-linear map $\varphi:U_0\to F^s$, define
\begin{equation}\label{eq:graph-stabilization}
 U_\varphi=
 \{(u,\varphi(u)+a):u\in U_0,\ a\in\F_q^s\}
 \leq F^{k_0+s}.
\end{equation}

\begin{proposition}\label{prop:graph-stabilization}
For every $s\geq1$ and every $\F_q$-linear map $\varphi$, the space
$U_\varphi$ is an $h$-scattered $[n_0+s,k_0+s]_{q^m/q}$
$q$-system.
\end{proposition}

\begin{proof}
The parametrization in \eqref{eq:graph-stabilization} is injective,
so $\dim_{\F_q}(U_\varphi)=n_0+s$.  Moreover, $U_\varphi$ contains
$\{0\}\times\F_q^s$, whose $F$-span is $\{0\}\times F^s$, and its
projection onto the first component is $U_0$.  Since $U_0$ spans
$F^{k_0}$, the $F$-span of $U_\varphi$ is $F^{k_0+s}$.

Let $L\leq F^{k_0+s}$ have dimension $h$, let
$V=\{0\}\times F^s$, and put $t=\dim_F(L\cap V)$.  The first
projection $\pi$ maps $L$ onto an $(h-t)$-dimensional $F$-subspace.
An $h$-scattered $q$-system is also $j$-scattered for every
$j\leq h$ by \cite[Proposition~2.1]{CMPZ}; hence
$$
 \dim_{\F_q}\pi(U_\varphi\cap L)
 \leq\dim_{\F_q}(U_0\cap\pi(L))\leq h-t.
$$
The kernel is contained in
$\{0\}\times(\F_q^s\cap(L\cap V))$.  Any vectors of $\F_q^s$
that are independent over $\F_q$ remain independent over $F$, because
matrix rank is preserved under the scalar extension $\F_q\subseteq F$.
Therefore this kernel has $\F_q$-dimension at most $t$.  Thus
$$
 \dim_{\F_q}(U_\varphi\cap L)\leq(h-t)+t=h.
$$
\end{proof}

If $1\leq s<(m-2)/2$, then
\begin{equation}\label{eq:graph-extremal-rank}
 \left\lfloor\frac{m(k_0+s)}{m-2}\right\rfloor
 =\ell m+s+\left\lfloor\frac{2s}{m-2}\right\rfloor
 =n_0+s.
\end{equation}
Consequently, every $U_\varphi$ in this range is extremal.  It is
quasi-maximum because
$m(k_0+s)\equiv2s\not\equiv0\pmod{m-2}$.  Taking $\varphi=0$
gives the elementary extension $U_0\oplus\F_q^s$.

We now isolate the additional condition that produces the strong
coding-theoretic properties.  Write the columns of $G_0$ as
$g_1,\ldots,g_{n_0}$ and let
$$
 \Phi=[\,\varphi(g_1)\ \cdots\ \varphi(g_{n_0})\,]
 \in F^{s\times n_0}.
$$
Then a basis matrix of $U_\varphi$ is
\begin{equation}\label{eq:graph-generator}
 G_\Phi=
 \begin{pmatrix}G_0&0\\ \Phi&I_s\end{pmatrix}.
\end{equation}
Put $C_\Phi=\operatorname{rowsp}_F(G_\Phi)$.  Direct orthogonality
gives
\begin{equation}\label{eq:graph-dual-code}
 D_\Phi:=C_\Phi^\perp
 =\{(x,-x\Phi^\top):x\in D_0\},
 \qquad
 H_\Phi=[\,H_0\mid -H_0\Phi^\top\,]
\end{equation}
as a generator matrix of $D_\Phi$.

Thus the coding operation is simply the column extension
$H_0\mapsto[\,H_0\mid Z\,]$, with $Z=-H_0\Phi^\top$; the graph
matrix records its dual $q$-system.  Maps whose matrices differ by
$LG_0$, with $L\in F^{s\times k_0}$, give the same $Z$ and are
removed by an $F$-linear shear of the graph coordinates.

Let $Z=-H_0\Phi^\top$ have columns $z_1,\ldots,z_s$, and set
$$
 \overline E_Z:=
 \langle z_1+W,\ldots,z_s+W\rangle_{\F_q}
 \leq F^{2\ell}/W.
$$
We shall use the following independence condition:
\begin{equation}\label{eq:independence-condition}
 \dim_{\F_q}\overline E_Z=s.
\end{equation}
Since the columns of $H_0$ form an $\F_q$-basis of $W$,
\eqref{eq:independence-condition} is equivalent to the
$\F_q$-linear independence of all columns of $H_\Phi$, and hence to
the nondegeneracy of $D_\Phi$.  Equivalently, if
$E_Z=\langle z_1,\ldots,z_s\rangle_{\F_q}$, then
$\dim_{\F_q}E_Z=s$ and $E_Z\cap W=\{0\}$.

Condition~\eqref{eq:independence-condition} can always be met in our
range.  Choose any $s$-dimensional subspace of $F^{2\ell}/W$, take a
basis, and lift its elements to columns $z_1,\ldots,z_s$ of $Z$.
Since $H_0$ has full row rank, one can solve
$-H_0\Phi^\top=Z$ for $\Phi$.

For explicit calculations, normalize
\begin{equation}\label{eq:normalized-initial-pair}
 H_0=[\,I_{2\ell}\mid P\,],
 \qquad
 G_0=[\,-P^\top\mid I_{k_0}\,].
\end{equation}
For any $Z\in F^{2\ell\times s}$, the choice
$\Phi=[\,-Z^\top\mid0\,]$ gives
\begin{equation}\label{eq:normalized-graph-pair}
 G_\Phi=
 \begin{pmatrix}
  -P^\top&I_{k_0}&0\\
  -Z^\top&0&I_s
 \end{pmatrix},
 \qquad
 H_\Phi=[\,I_{2\ell}\mid P\mid Z\,].
\end{equation}

Thus \eqref{eq:independence-condition} identifies exactly the column
extensions for which $D_\Phi$ is nondegenerate, while
\eqref{eq:normalized-graph-pair} displays the corresponding dual
graph systems explicitly.

\begin{theorem}\label{thm:graph}
Let $1\leq s<(m-2)/2$.  For every $\varphi$, the system $U_\varphi$
is extremal $(m-3)$-scattered and quasi-maximum.  Its associated code
$C_\Phi$ is nondegenerate, while $D_\Phi=C_\Phi^\perp$ is quasi-MRD
with parameters
$$
 [\ell m+s,2\ell,m-1]_{q^m/q}.
$$
Moreover,
\begin{equation}\label{eq:graph-weight-monotonicity}
 d_i(D_\Phi)\geq d_i(D_0)\qquad(1\leq i\leq2\ell).
\end{equation}
This inequality holds for every $\varphi$, independently of
\eqref{eq:independence-condition}.  If
\eqref{eq:independence-condition} holds, then both $C_\Phi$ and
$D_\Phi$ are nondegenerate quasi-MRD codes, with
\begin{equation}\label{eq:qMRD-parameters}
 \begin{split}
 C_\Phi&:\ [\ell m+s,\ell(m-2)+s,2]_{q^m/q},\\
 D_\Phi&:\ [\ell m+s,2\ell,m-1]_{q^m/q}.
 \end{split}
\end{equation}
If \eqref{eq:independence-condition} fails, then $d(C_\Phi)=1$.
\end{theorem}

\begin{proof}
The geometric assertions follow from
Proposition~\ref{prop:graph-stabilization} and
\eqref{eq:graph-extremal-rank}.  Puncturing the last $s$ coordinates
maps $D_\Phi$ isomorphically onto $D_0$, so
$d(D_\Phi)\geq m-1$.  Since
$$
 1<\frac{2\ell m}{\ell m+s}<2,
$$
the second term of the Singleton bound gives $d(D_\Phi)\leq m-1$.
Thus equality holds and
$$
 \Rdef(D_\Phi)=m-2-(m-1)+1=0.
$$
For an $i$-dimensional subcode $V\leq D_\Phi$, puncturing produces
an $i$-dimensional subcode of $D_0$ whose rank support is a coordinate
projection of that of $V$.  This proves
\eqref{eq:graph-weight-monotonicity}.

The columns of $G_\Phi$ are $\F_q$-independent: otherwise a nonzero
$\F_q$-relation among them would be a rank-one word of $D_\Phi$,
contrary to $d(D_\Phi)=m-1>1$.  Hence $C_\Phi$ is nondegenerate and
its column system is $U_\varphi$.

If \eqref{eq:independence-condition} holds, a rank-one word of
$C_\Phi$ would, after scaling, be a nonzero $\F_q$-relation among
the columns of $H_\Phi$.  Therefore $d(C_\Phi)\geq2$.  The
Singleton bound gives $d(C_\Phi)\leq2$, and hence equality.  Since
$$
 \left\lceil
 \frac{m(\ell(m-2)+s)}{\ell m+s}
 \right\rceil=m-1,
$$
the code $C_\Phi$ has rank defect zero.  If
\eqref{eq:independence-condition} fails, a nonzero
$\F_q$-relation among the columns of $H_\Phi$ is a rank-one word of
$C_\Phi$, so $d(C_\Phi)=1$.
\end{proof}

We now count equivalence classes, rather than presentations by graph
maps.  Recall that
$$
 \qbinom{N}{s}
 =\prod_{i=0}^{s-1}\frac{q^{N-i}-1}{q^{s-i}-1}
$$
is the Gaussian coefficient.

Suppose that $Z$ satisfies \eqref{eq:independence-condition}, and let
$S_Z=W+E_Z$ be the $\F_q$-span of the columns of
$[\,H_0\mid Z\,]$.  At the level of column systems,
$S_Z$ depends only on the subspace
$\overline E_Z=S_Z/W\leq F^{2\ell}/W$.  Conversely, every
$s$-dimensional subspace of this quotient has a unique inverse image
$S$ containing $W$, and any basis of $S/W$, after choosing lifts,
produces $S$.  Thus changes of basis and of lifts are not counted as
distinct extensions.

\begin{theorem}\label{thm:orbit-count}
Fix a maximum scattered $q$-system $W\leq F^{2\ell}$ and let
$1\leq s<(m-2)/2$.  Among the extensions satisfying
\eqref{eq:independence-condition}, the construction produces at least
\begin{equation}\label{eq:orbit-lower-bound}
 \left\lceil
 \frac{\qbinom{\ell m}{s}}
 {em\,|\GL(2\ell,F)|}
 \right\rceil
\end{equation}
pairwise inequivalent extremal $(m-3)$-scattered $q$-systems.
\end{theorem}

\begin{proof}
Let $\mathcal X_s(W)$ be the set of $\F_q$-subspaces $S$ such that
$$
 W\leq S\leq F^{2\ell},
 \qquad \dim_{\F_q}S=\ell m+s.
$$
The quotient $F^{2\ell}/W$ has $\F_q$-dimension $\ell m$, and hence
$|\mathcal X_s(W)|=\qbinom{\ell m}{s}$.  Every $S$ in this set spans
$F^{2\ell}$ over $F$, since it contains $W$.  Choose a basis of
$S/W$, lift it to columns $Z$, and append these columns to $H_0$.
Then \eqref{eq:independence-condition} holds and
$[\,H_0\mid Z\,]$ has column system $S$, so every member of
$\mathcal X_s(W)$ occurs in the construction.

If two output graph systems are equivalent, the code--system
correspondence and Lemma~\ref{lem:equivalence-duality} imply that
their column systems $S$ are semilinearly equivalent.  The group
$\GL(2\ell,F)\rtimes\operatorname{Aut}(F)$ has order
$em\,|\GL(2\ell,F)|$, so each equivalence class contains at most this
many members of $\mathcal X_s(W)$.  Dividing gives
\eqref{eq:orbit-lower-bound}.  We deliberately divide by the full
semilinear group, rather than analyze the stabilizer of $W$; the
estimate is uniform and asymptotically nontrivial, although it does
not exploit the stabilizer of $W$.
The set $\mathcal X_s(W)$ need not be invariant under the full group;
this causes no problem, since only the universal upper bound on the
size of its intersection with an equivalence class is used.
\end{proof}

\begin{corollary}\label{cor:orbit-asymptotic}
Fix $q$ and $\ell\geq1$.  For $m\geq5$, put
$s_m=\lfloor(m-3)/2\rfloor$.  The lower bound in
Theorem~\ref{thm:orbit-count}, and hence the number of inequivalent
systems produced for each $m$, tends to infinity as $m\to\infty$.
\end{corollary}

\begin{proof}
Maximum scattered systems exist for every $m>1$; for instance one may
take
$$
 \{(x_1,\ldots,x_\ell,x_1^q,\ldots,x_\ell^q):
     x_1,\ldots,x_\ell\in F\}.
$$
Indeed, on every nonzero $F$-line, two nonzero vectors of this set
can differ only by a scalar $\lambda$ satisfying
$\lambda^q=\lambda$, while each coordinate graph spans $F^2$.
Moreover,
$$
 \qbinom{\ell m}{s_m}\geq q^{s_m(\ell m-s_m)},
 \qquad
 |\GL(2\ell,F)|<q^{4\ell^2m}.
$$
Therefore the lower bound in \eqref{eq:orbit-lower-bound} is at least
$$
 \frac{q^{s_m(\ell m-s_m)-4\ell^2m}}{em},
$$
whose exponent has leading term $(2\ell-1)m^2/4$.
\end{proof}

We next compare the second generalized rank weight of $D_\Phi$ with
a precisely specified direct-sum class.  We say that a code $D$ has
a \emph{nontrivial direct-sum decomposition} if it is rank-metric
equivalent to $D_1\oplus D_2$ on two disjoint nonempty coordinate
blocks, where both $D_1$ and $D_2$ have positive dimension.  A
decomposition into more than two positive-dimensional factors can be
grouped into one of this form.  This notion is broader than, and
should not be confused with, complete decomposability in the sense of
\cite{Santonastaso}, where all factors are one-dimensional MRD codes.

Let $\mathbf{s}=(s_1,\ldots,s_\ell)$ be a tuple of nonnegative
integers with $s_1+\cdots+s_\ell=s$.  An \emph{order-one
direct-sum code} is
\begin{equation}\label{eq:order-one-block-code}
 D_{\mathbf{s}}=D_{s_1}\oplus\cdots\oplus D_{s_\ell}
\end{equation}
on disjoint coordinate blocks, where every factor is a nondegenerate
$$
 [m+s_i,2,m-1]_{q^m/q}
$$
code.  These factors exist: use the Gabidulin code when $s_i=0$ and,
when $s_i>0$, an extension from Theorem~\ref{thm:graph} with
$\ell=1$ satisfying \eqref{eq:independence-condition}.  The code
$D_{\mathbf{s}}$ is nondegenerate with parameters
$[\ell m+s,2\ell,m-1]_{q^m/q}$.

\begin{proposition}
\label{prop:direct-sum-separation}
Assume $\ell\geq2$ and $1\leq s<(m-2)/2$.  Then every code in
\eqref{eq:order-one-block-code} satisfies
\begin{equation}\label{eq:block-d2}
 d_2(D_{\mathbf{s}})=m+\min_i s_i
 \leq m+\left\lfloor\frac{s}{\ell}\right\rfloor.
\end{equation}

Let $D_0$ be the MRD code associated with $W$, let $\varphi$ satisfy
\eqref{eq:independence-condition}, and put $D=D_\Phi$.  If
\begin{equation}\label{eq:separation-criterion}
 d_2(D_0)>m+\left\lfloor\frac{s}{\ell}\right\rfloor,
\end{equation}
then $U_\varphi$ is not equivalent to any $q$-system associated with
$D_{\mathbf{s}}^\perp$.  The weaker inequality $d_2(D_0)>m$ gives
the same conclusion for every tuple $\mathbf{s}$ having at least one
$s_i=0$.
\end{proposition}

\begin{proof}
Choose $j$ with $s_j=\min_i s_i$.  Since $D_{s_j}$ is
two-dimensional and nondegenerate, its second generalized rank
weight is its length, namely $m+s_j$.  Viewing this factor as a
subcode of $D_{\mathbf{s}}$ gives
$d_2(D_{\mathbf{s}})\leq m+s_j$.

For the reverse inequality, let $V\leq D_{\mathbf{s}}$ have
$F$-dimension two and denote its projection onto the $i$-th block by
$V_i$.  If $\dim_F V_i=2$ for some $i$, puncturing to that block
 gives
$$
 \dim_{\F_q}\operatorname{Rsupp}_q(V)
 \geq d_2(D_{s_i})=m+s_i\geq m+s_j.
$$
Otherwise all the $V_i$ have dimension at most one.  Since the joint
block-projection map is injective on $V$, the kernels of its component
maps have trivial intersection.  Hence two components, say those in
blocks $a$ and $b$, already give a map $V\to V_a\oplus V_b$ of rank
two.  This map is an isomorphism, so puncturing to these two blocks
gives $V_a\oplus V_b$.  Because the two coordinate blocks are
disjoint,
$$
 \dim_{\F_q}\operatorname{Rsupp}_q(V_a\oplus V_b)
 =\dim_{\F_q}\operatorname{Rsupp}_q(V_a)
  +\dim_{\F_q}\operatorname{Rsupp}_q(V_b)
 \geq2(m-1),
$$
where the equality is valid because the projected subcode is the full
external product $V_a\oplus V_b$, not merely an arbitrary subcode of
a concatenation.  The last inequality follows from
$d(D_{s_a})=d(D_{s_b})=m-1$.
Since $s_j\leq s<(m-2)/2<m-2$, we have
$m+s_j<2(m-1)$.  Thus every such $V$ has rank support at least
$m+s_j$, proving the equality in \eqref{eq:block-d2}.  The displayed
upper bound follows from
$\min_i s_i\leq\lfloor s/\ell\rfloor$.

By \eqref{eq:graph-weight-monotonicity},
$d_2(D)\geq d_2(D_0)$.  Under
\eqref{eq:separation-criterion}, this is strictly larger than
$d_2(D_{\mathbf{s}})$ for every tuple $\mathbf{s}$.  If the output
$q$-system were equivalent to the dual system of some
$D_{\mathbf{s}}$, the code--system correspondence and
Lemma~\ref{lem:equivalence-duality} would make $D$ equivalent to
$D_{\mathbf{s}}$, contradicting invariance of $d_2$.  Finally, if
some $s_i=0$, then \eqref{eq:block-d2} gives
$d_2(D_{\mathbf{s}})=m$, so the same argument applies under the
weaker assumption $d_2(D_0)>m$.
\end{proof}

\begin{proposition}\label{prop:all-decompositions-order-two}
Let $1\leq s<(m-2)/2$, and let $D$ be a nondegenerate
$[2m+s,4,m-1]_{q^m/q}$ code.  If $D$ has a nontrivial direct-sum
decomposition, then it is rank-metric equivalent to
$$
 D_{s_1}\oplus D_{s_2},
 $$
where $s_1,s_2\geq0$, $s_1+s_2=s$, and each $D_{s_i}$ is a
nondegenerate $[m+s_i,2,m-1]_{q^m/q}$ code.  Consequently, if
\begin{equation}\label{eq:full-direct-sum-separation}
 d_2(D)>m+\left\lfloor\frac{s}{2}\right\rfloor,
\end{equation}
then $D$ is not rank-metric equivalent to any nontrivial direct sum.
\end{proposition}

\begin{proof}
Suppose that $D$ is equivalent to $D_1\oplus D_2$, and write
$$
 r_i=\dim_F D_i,\qquad n_i=\text{length}(D_i)\qquad(i=1,2).
$$
Thus $r_1+r_2=4$ and $n_1+n_2=2m+s$.  Since the minimum distance of
a direct sum is the minimum of the distances of its factors, we have
$d(D_i)\geq m-1$ for both $i$.  The two terms of the Singleton
bound~\eqref{eq:singleton} therefore imply
\begin{equation}\label{eq:factor-length-bounds}
 n_i\geq m+r_i-2,
 \qquad
 n_i\geq\left\lceil\frac{mr_i}{2}\right\rceil.
\end{equation}
Nondegeneracy is preserved by rank-metric equivalence, and a direct
sum is nondegenerate precisely when each factor is nondegenerate.

Up to interchanging the factors, the possible dimension splits are
$(r_1,r_2)=(1,3)$ and $(2,2)$.  In the first case,
\eqref{eq:factor-length-bounds} gives
$$
 n_1+n_2\geq m-1+\left\lceil\frac{3m}{2}\right\rceil.
$$
On the other hand, $s<(m-2)/2$ gives
$$
 2m+s<\frac{5m}{2}-1
 \leq m-1+\left\lceil\frac{3m}{2}\right\rceil,
$$
a contradiction.  Hence $r_1=r_2=2$.  The first inequality in
\eqref{eq:factor-length-bounds} gives $n_i\geq m$, so we may write
$n_i=m+s_i$ with $s_i\geq0$ and $s_1+s_2=s$.

It remains to determine the distances of the two factors.  If
$d(D_i)\geq m$, the second term of the Singleton bound would give
$$
 2m\leq n_i\bigl(m-d(D_i)+1\bigr)\leq n_i.
$$
But the other factor has length at least $m$, and hence
$$
 n_i\leq m+s<2m,
$$
which is impossible.  Thus $d(D_i)=m-1$ for $i=1,2$, proving the
first assertion.

Every possible nontrivial decomposition therefore belongs to the
class in Proposition~\ref{prop:direct-sum-separation} with
$\ell=2$.  Its second generalized rank weight is at most
$m+\lfloor s/2\rfloor$, so \eqref{eq:full-direct-sum-separation}
excludes all such decompositions.
\end{proof}

\section{Extensions of the Bartoli--Marino--Neri--Vicino systems}
\label{sec:BMNV}

We now specialize Theorem~\ref{thm:graph} to a concrete family of
maximum scattered systems in $F^4$.

Following the range stated in \cite[Definition~3.1]{BMNV}, let
$I,J$ be distinct integers with
$$
 1\leq I,J<m-1,
$$
and let
$\alpha,\beta,\gamma\in F^*$.  Define
\begin{equation}\label{eq:BMNV-W}
 W_{\alpha,\beta,\gamma}^{I,J}
 =\{(x,y,x^{q^I}+\alpha y^{q^J},
           x^{q^J}+\beta y^{q^I}+\gamma y^{q^J}):x,y\in F\}.
\end{equation}
Associated with these parameters is the projective polynomial
\begin{equation}\label{eq:BMNV-P}
 P_{\alpha,\beta,\gamma}^{I,J}(X)=
 \begin{cases}
 X^{q^{J-I}+1}+\gamma X-\alpha\beta,&I<J,\\
 X^{q^{I-J}+1}+\gamma X^{q^{I-J}}-\alpha\beta,&I>J.
 \end{cases}
\end{equation}

We use the following results of \cite[Theorems~3.2 and~3.5,
Remark~3.8]{BMNV}.

\begin{theorem}\label{thm:BMNV}
Assume
$$
 \gcd(I,J,m)=1
$$
and that $P_{\alpha,\beta,\gamma}^{I,J}$ has no root in $F$.
Then $W_{\alpha,\beta,\gamma}^{I,J}$ is a maximum scattered
$[2m,4]_{q^m/q}$ $q$-system.  If $D_0$ is an associated code and
$M=\max\{I,J\}\leq(m-1)/2$, then
\begin{equation}\label{eq:BMNV-d2}
 d_2(D_0)\geq2(m-M)\geq m+1.
\end{equation}
\end{theorem}

Choose an $s$-dimensional $\F_q$-subspace $B\leq F$ and bases
$$
 F=\langle\xi_1,\ldots,\xi_m\rangle_{\F_q},
 \qquad
 B=\langle b_1,\ldots,b_s\rangle_{\F_q},
$$
and define columns in $F^4$ by
\begin{align}
 c_i&=(\xi_i,0,\xi_i^{q^I},\xi_i^{q^J})^\top,
 \nonumber\\
 d_i&=(0,\xi_i,\alpha\xi_i^{q^J},
              \beta\xi_i^{q^I}+\gamma\xi_i^{q^J})^\top,
 \label{eq:BMNV-columns}\\
 z_j&=(0,0,b_j,0)^\top.\nonumber
\end{align}
Let
\begin{equation}\label{eq:BMNV-H}
 H_0=[\,c_1\ \cdots\ c_m\ d_1\ \cdots\ d_m\,],
 \qquad Z_B=[\,z_1\ \cdots\ z_s\,],
 \qquad H_B=[\,H_0\mid Z_B\,].
\end{equation}
If $\sum_j\lambda_jz_j\in W_{\alpha,\beta,\gamma}^{I,J}$ with
$\lambda_j\in\F_q$, then its first two coordinates force $x=y=0$.
Hence $\sum_j\lambda_jb_j=0$, and therefore every $\lambda_j$ is
zero.  Thus $\dim_{\F_q}\overline E_{Z_B}=s$, so
\eqref{eq:independence-condition} holds.  Choose $\Phi_B$ satisfying
$-H_0\Phi_B^\top=Z_B$ and an invertible four-column minor $A$
of $H_0$, move it first, write $H_B=[A\ Q]$, and put
\begin{equation}\label{eq:BMNV-G}
 G_B=[\,-Q^\top A^{-\top}\;I_{2m+s-4}\,].
\end{equation}
This is a basis matrix of the dual graph system $U_{\varphi_B}$.

\begin{theorem}\label{thm:main-BMNV}
Assume the hypotheses of Theorem~\ref{thm:BMNV} and let
$$
 1\leq s<\frac{m-2}{2}.
$$
Then the columns of $G_B$ in \eqref{eq:BMNV-G} span an extremal
$(m-3)$-scattered
\begin{equation}\label{eq:BMNV-output-parameters}
 [2m+s,\,2m+s-4]_{q^m/q}
\end{equation}
$q$-system.  Put $D_B:=\operatorname{rowsp}_F(H_B)$.  This is the
dual of the code associated with the constructed $q$-system, and it
is a nondegenerate quasi-MRD code with parameters
$$
 [2m+s,4,m-1]_{q^m/q}.
$$
The code associated with the constructed system is a nondegenerate
$[2m+s,2m+s-4,2]_{q^m/q}$ quasi-MRD code.
If $M:=\max\{I,J\}\leq(m-1)/2$, then
$$
 d_2(D_B)\geq2(m-M)\geq m+1.
$$
If, in addition,
\begin{equation}\label{eq:BMNV-all-block-separation}
 2(m-M)>m+\left\lfloor\frac{s}{2}\right\rfloor,
\end{equation}
then the system is inequivalent to every $q$-system associated with
$D_{(s_1,s_2)}^\perp$, where $s_1+s_2=s$ and
$D_{(s_1,s_2)}$ is an order-one direct-sum code.  Under the same
hypothesis, $D_B$ is not rank-metric equivalent to any nontrivial
direct sum.
\end{theorem}

\begin{proof}
By Theorem~\ref{thm:BMNV},
$W_{\alpha,\beta,\gamma}^{I,J}$ is maximum scattered of rank $2m$
in $F^4$, and the preceding argument verifies
\eqref{eq:independence-condition}.  Theorem~\ref{thm:graph}
therefore gives
$$
 n=2m+s,
 \qquad k=2(m-2)+s=2m+s-4,
$$
together with the claimed scatteredness and code parameters.

Let $D_0$ be the code associated with
$W_{\alpha,\beta,\gamma}^{I,J}$.  By
\eqref{eq:graph-weight-monotonicity} and \eqref{eq:BMNV-d2},
$$
 d_2(D_B)\geq d_2(D_0)\geq2(m-M).
$$
Finally, for $\ell=2$ the right-hand side of
\eqref{eq:separation-criterion} is
$m+\lfloor s/2\rfloor$.  Thus
\eqref{eq:BMNV-all-block-separation} is precisely the hypothesis
needed in Proposition~\ref{prop:direct-sum-separation}.
The same strict inequality and
Proposition~\ref{prop:all-decompositions-order-two} prove the final
assertion.
\end{proof}

\begin{corollary}\label{cor:BMNV-threshold}
Let $m\geq5$ be odd and put $s=(m-3)/2$.  Under the hypotheses of
Theorem~\ref{thm:BMNV}, the matrix $G_B$ defines an extremal
$(m-3)$-scattered
$$
 \left[\frac{5m-3}{2},\frac{5m-11}{2}\right]_{q^m/q}
$$
$q$-system.  If $M=\max\{I,J\}\leq(m-1)/2$, the code
$D_B=\operatorname{rowsp}_F(H_B)$ has $d_2(D_B)\geq2(m-M)$.  If
$$
 2(m-M)>m+\left\lfloor\frac{m-3}{4}\right\rfloor,
$$
then the constructed system is inequivalent to every $q$-system associated with
$D_{(s_1,s_2)}^\perp$, where $s_1+s_2=(m-3)/2$ and
$D_{(s_1,s_2)}$ is an order-one direct-sum code.  Under the same
hypothesis, the code $D_B$ is not rank-metric equivalent to any
nontrivial direct sum.
\end{corollary}

\begin{proof}
For odd $m\geq5$, the integer $s=(m-3)/2$ satisfies
$1\leq s<(m-2)/2$.  Substitution in
\eqref{eq:BMNV-output-parameters} gives
$$
 2m+s=\frac{5m-3}{2},
 \qquad
 2m+s-4=\frac{5m-11}{2}.
$$
Moreover,
$\lfloor s/2\rfloor=\lfloor(m-3)/4\rfloor$, so both conclusions
under the displayed inequality follow from
Theorem~\ref{thm:main-BMNV} with this value of $s$.
\end{proof}

The hypotheses are nonempty in an infinite range.  One convenient
specialization follows from \cite[Remark~3.3]{BMNV}.  Take
$I=1$, $J=2$, and choose $\alpha,\beta,\gamma\in\F_q^*$ so that
$$
 X^2+\gamma X-\alpha\beta
$$
is primitive over $\F_q$.  If $(q+1)\nmid m$, then
$P_{\alpha,\beta,\gamma}^{1,2}$ has no root in $\F_{q^m}$.
Consequently Corollary~\ref{cor:BMNV-threshold} gives infinitely many
extension degrees $m$ for every fixed $q$.  Here $M=2$, and the
strict generalized-weight inequality in that corollary holds for
every odd $m\geq5$, because
$$
 2(m-2)-m-\left\lfloor\frac{m-3}{4}\right\rfloor
 \geq\frac{3m-13}{4}>0.
$$

\subsection{A fully explicit binary example}

Let $q=2$, $m=5$, and
$$
 F=\F_{32}=\F_2[X]/(X^5+X^2+1),
 \qquad \omega=X\bmod (X^5+X^2+1).
$$
The polynomial $X^5+X^2+1$ is irreducible over $\F_2$: it has no
root in $\F_2$, and its remainder modulo the unique irreducible
quadratic $X^2+X+1$ is $1$.  A reducible polynomial of degree five
without a linear factor would have an irreducible quadratic factor.
Thus $\omega^5=\omega^2+1$.
Take $I=1$, $J=2$, $\alpha=\beta=\gamma=1$, and $B=\F_2$.
The polynomial in \eqref{eq:BMNV-P} is
$$
 P(X)=X^3+X+1.
$$
It has no root in $\F_{32}$: it is irreducible of degree three over
$\F_2$, while $\F_8$ is not a subfield of $\F_{32}$ because
$3\nmid5$.  Thus the corresponding $W_{1,1,1}^{1,2}$ is maximum
scattered.  For $B=\F_2$, the appended column
$z=(0,0,1,0)^\top$ does not belong to
$W_{1,1,1}^{1,2}$, because its first two coordinates would force
$x=y=0$.  Hence $H_B=[H_0\mid z]$
satisfies \eqref{eq:independence-condition}.

Use the power basis $\mathcal B=(1,\omega,\ldots,\omega^4)$ in
\eqref{eq:BMNV-columns}, and order the resulting columns as
$$
 c_1,c_2,c_3,d_1,c_4,c_5,d_2,d_3,d_4,d_5,z.
$$
Magma computations over $F$ show that the first four columns form the
invertible minor
$$
 A=[\,c_1\ c_2\ c_3\ d_1\,]
 =\begin{pmatrix}
 1&\omega&\omega^2&0\\
 0&0&0&1\\
 1&\omega^2&\omega^4&1\\
 1&\omega^4&\omega^8&0
\end{pmatrix}.
$$
The same computations show that the systematic matrix in
\eqref{eq:BMNV-G} is
\begin{equation}\label{eq:binary-G}
 G=[\,R\;I_7\,],
\end{equation}
where the entries are displayed in the basis $\mathcal B$,
{\small
$$
R=\begin{pmatrix}
\omega^4+\omega^2 & \omega^2+\omega+1 &
 \omega^4+\omega^2+\omega & 0\\
\omega^4 & \omega^4+\omega^3+1 &
 \omega^4+\omega^3+\omega^2+\omega & 0\\
\omega^3+\omega & \omega^2 & \omega^4+\omega & \omega\\
\omega^3+\omega & \omega^3+\omega^2+\omega+1 &
 \omega^2+1 & \omega^2\\
\omega^4+\omega^3+\omega+1 & \omega^4 &
 \omega^4+\omega^2+1 & \omega^3\\
\omega^4+\omega+1 & \omega^4 &
 \omega^4+\omega^2+\omega+1 & \omega^4\\
\omega^4+1 & \omega^4 & \omega^2+1 & 0
\end{pmatrix}.
$$
}

\begin{corollary}\label{cor:binary}
The $\F_2$-span of the eleven columns of the matrix $G$ in
\eqref{eq:binary-G} is an extremal $2$-scattered
$[11,7]_{2^5/2}$ $q$-system.  Let $C=\operatorname{rowsp}_F(G)$.
Then $C$ is a nondegenerate $[11,7,2]_{2^5/2}$ quasi-MRD code and
$$
 D:=C^\perp=\operatorname{rowsp}_F(H_B)
$$
is a nondegenerate $[11,4,4]_{2^5/2}$ quasi-MRD code, and
$$
 d_2(D)=6.
$$
Its complete generalized rank-weight hierarchy is
$$
 (d_1(D),d_2(D),d_3(D),d_4(D))=(4,6,9,11).
$$
Consequently, the constructed $q$-system is inequivalent to every
$q$-system associated with $D_{(s_1,s_2)}^\perp$, where
$s_1+s_2=1$ and $D_{(s_1,s_2)}$ is an order-one direct-sum code.
Furthermore, $D$ is not rank-metric equivalent to any nontrivial
direct sum.
\end{corollary}

\begin{proof}
Magma computations give
$\det(A)=\omega^4+\omega^3+1\neq0$, and the definition
$R=-Q^\top A^{-\top}$ gives $GH_B^\top=0$.  Since $H_B$ has row rank
four and $G$ has row rank seven,
$\operatorname{rowsp}_F(G)=D^\perp$.  Hence
Theorem~\ref{thm:main-BMNV} applies.  Here
$m-3=2$, $s=1$, and
$$
 \left\lfloor\frac{7\cdot5}{3}\right\rfloor=11.
$$
Finally, \eqref{eq:BMNV-d2} gives $d_2(D)\geq2(5-2)=6$.
For the reverse inequality, let $r_1,\ldots,r_4$ be the rows of
$H_B$ and set
$$
 v_1=r_1+r_3+r_4,
 \qquad v_2=r_2+r_3.
$$
Their coefficient vectors are independent, so they span a
two-dimensional subcode of $D$.  Expanding $v_1,v_2$ in the basis
$\mathcal B$, Magma gives rank six for the resulting binary
coefficient matrix.  This rank is precisely
$\dim_{\F_2}\operatorname{Rsupp}_2
(\langle v_1,v_2\rangle_F)$.
Thus $d_2(D)\leq6$, and hence $d_2(D)=6$.
An exhaustive Magma enumeration of the one-, two-, and
three-dimensional $F$-subcodes gives
$(d_1(D),d_2(D),d_3(D))=(4,6,9)$; nondegeneracy gives
$d_4(D)=11$.
Finally, Proposition~\ref{prop:direct-sum-separation} gives
$d_2(D_{(s_1,s_2)})=5$ whenever $s_1+s_2=1$, proving the last
inequivalence assertion.  Since $d_2(D)=6>5$, Proposition~\ref{prop:all-decompositions-order-two} also excludes every
nontrivial direct-sum decomposition of $D$.
\end{proof}

\section{Extensions in higher order}\label{sec:higher-order}

The extension theorem is not restricted to order two.  We record its
direct application to the higher-order scattered sequences of
\cite{BGM}.  The conclusion below separates the outputs from the
specified order-one direct-sum class; it does not assert
indecomposability against arbitrary dimension splits.

In the notation of \cite{BGM}, the order denoted there by $m$ is our
$\ell$, while its extension degree $n$ is our $m$; the symbols
$A,I,J$, and $K=J-I$ are unchanged.
Let $\ell\geq3$, let $A=(\alpha_1,\ldots,\alpha_\ell)\in(F^*)^\ell$,
and let $0<I<J<m$ with $\gcd(I,J)=1$.  Put
$$
 K=J-I,
 \qquad
 C_{a,r}=\frac{q^{ar}-1}{q^a-1}.
$$
Read subscripts cyclically modulo $\ell$, so that, for example,
$\alpha_{\ell+1}=\alpha_1$.  Define
\begin{align*}
 f_i(x_1,\ldots,x_\ell)
 &=x_i^{q^I}+\alpha_{i+1}x_{i+1}^{q^J}
 \quad(1\leq i<\ell),\\
 f_\ell(x_1,\ldots,x_\ell)
 &=x_\ell^{q^I}+\alpha_1x_1^{q^J},
\end{align*}
and
\begin{equation}\label{eq:BGM-system}
 \begin{split}
 W_A^{I,J}=\{&(x_1,\ldots,x_\ell,
                   f_1(x),\ldots,f_\ell(x)):\\
              &x_1,\ldots,x_\ell\in F\}\leq F^{2\ell}.
 \end{split}
\end{equation}
Define
$$
 \kappa_A^{I,J}:=
 \frac{\alpha_1^{q^{(\ell-2)K}}
       \displaystyle\prod_{u=3}^{\ell}
       \alpha_u^{q^{(u-3)K}}}
      {\alpha_2^{C_{K,\ell-1}}}
$$
and, for cyclic indices, put
$$
 \Pi_i:=
 \alpha_i^{q^{(\ell-1)K}}
 \alpha_{i-1}^{q^{(\ell-2)K}}\cdots
 \alpha_{i+2}^{q^K}\alpha_{i+1}.
$$
For a positive integer $e$, an element $c\in F^*$ is called an
$e$-th power in $F$ if $c=v^e$ for some $v\in F^*$.
By \cite[Theorem~2.4]{BGM}, the space $W_A^{I,J}$ is maximum
scattered whenever $\kappa_A^{I,J}$ is not a
$C_{K,\ell}$-th power in $F$.  If, in addition, there exists
$\delta\in\{1,\ldots,\ell-1\}$ such that
\begin{equation}\label{eq:BGM-power-condition}
 \frac{\Pi_{\delta+2}}{\Pi_2}
 \quad\text{is not a $(q^{\ell K}-1)$-st power in $F$},
\end{equation}
then \cite[Theorem~2.15 and Remark~2.17]{BGM} show that an associated
MRD code $D_0$ satisfies
\begin{equation}\label{eq:BGM-d2}
 d_2(D_0)\geq2m-2J.
\end{equation}
In particular, if $m\geq2J+1$, then $d_2(D_0)>m$.

Let $H_0$ be a basis matrix of $W_A^{I,J}$.  Choose an
$s$-dimensional $\F_q$-subspace $B\leq F$ with basis
$b_1,\ldots,b_s$, and let $Z_B$ have columns
$(0,\ldots,0,b_j,0,\ldots,0)^\top$, with $b_j$ in the position of
$f_1$.  If an $\F_q$-linear combination of these columns belongs to
$W_A^{I,J}$, its first $\ell$ coordinates force
$x_1=\cdots=x_\ell=0$; the independence of the $b_j$ then forces all
coefficients to vanish.  Hence $\dim_{\F_q}\overline E_{Z_B}=s$, and
$$
 H_{A,B}=[\,H_0\mid Z_B\,]
$$
satisfies \eqref{eq:independence-condition}.  Let $G_{A,B}$ be any
generator matrix of $\operatorname{rowsp}_F(H_{A,B})^\perp$.

\begin{corollary}\label{cor:higher-order}
Suppose $\kappa_A^{I,J}$ is not a $C_{K,\ell}$-th power in $F$ and
that \eqref{eq:BGM-power-condition} holds.  Let
$1\leq s<(m-2)/2$.  The columns of $G_{A,B}$ span an extremal
$(m-3)$-scattered
$$
 [\ell m+s,\ell(m-2)+s]_{q^m/q}
$$
$q$-system.  Put $D:=\operatorname{rowsp}_F(H_{A,B})$.  Then $D$ is
the dual of the code
associated with the constructed $q$-system and satisfies
$$
 d_2(D)\geq2m-2J.
$$
If
\begin{equation}\label{eq:BGM-direct-sum-separation}
 m-2J>\left\lfloor\frac{s}{\ell}\right\rfloor,
\end{equation}
then the system is inequivalent to every $q$-system associated with
$D_{\mathbf{s}}^\perp$, where $\mathbf{s}$ has $\ell$ entries with
sum $s$ and $D_{\mathbf{s}}$ is an order-one direct-sum code.
\end{corollary}

\begin{proof}
The hypothesis on $\kappa_A^{I,J}$ makes $W_A^{I,J}$ maximum
scattered, and the preceding argument verifies
\eqref{eq:independence-condition}.  Theorem~\ref{thm:graph} yields
the stated parameters.  Moreover,
\eqref{eq:graph-weight-monotonicity} and \eqref{eq:BGM-d2} give
$$
 d_2(D)\geq d_2(D_0)\geq2m-2J.
$$
Finally, \eqref{eq:BGM-direct-sum-separation} is equivalent to
$$
 2m-2J>m+\left\lfloor\frac{s}{\ell}\right\rfloor.
$$
Thus Proposition~\ref{prop:direct-sum-separation} proves the
inequivalence assertion.
\end{proof}

\begin{corollary}\label{cor:higher-order-infinite}
Let $q$ be any prime power, let $\ell\geq3$ and $t\geq2$, and put
$m=\ell t$, $I=1$, and $J=2$.
There exists $A=(\alpha_1,\ldots,\alpha_\ell)\in(F^*)^\ell$ such
that, for every $1\leq s<(m-2)/2$ and every $s$-dimensional
$\F_q$-subspace $B\leq F$, Corollary~\ref{cor:higher-order} gives an extremal
$(m-3)$-scattered
$$
 [\ell m+s,\ell(m-2)+s]_{q^m/q}
$$
$q$-system whose dual code $D$ satisfies $d_2(D)\geq2m-4$ and is not
equivalent to any order-one direct-sum code with total surplus $s$.
\end{corollary}

\begin{proof}
Here $K=1$ and $\ell\nmid K$.  For fixed
$\alpha_2,\ldots,\alpha_\ell$, let $\mathcal B_0$ be the choices of
$\alpha_1$ for which $\kappa_A^{1,2}$ is a $C_{1,\ell}$-th power,
and let $\mathcal B_\delta$ be those for which
$\Pi_{\delta+2}/\Pi_2$ is a $(q^\ell-1)$-st power
$(1\leq\delta<\ell)$.  The proof of
\cite[Proposition~2.13]{BGM} counts these bad sets after fixing
$\alpha_2,\ldots,\alpha_\ell$ and gives, uniformly,
$$
 \left|\mathcal B_0\cup
 \bigcup_{\delta=1}^{\ell-1}\mathcal B_\delta\right|<q^m-1.
$$
Thus one choice of $\alpha_1$ avoids all conditions simultaneously.
Corollary~\ref{cor:higher-order} gives
$d_2(D)\geq2m-4$; moreover,
$$
 \left\lfloor\frac{s}{\ell}\right\rfloor
 <\frac{m-2}{2\ell}
 \leq\frac{m-2}{6}<m-4.
$$
Thus \eqref{eq:BGM-direct-sum-separation} is automatic.
\end{proof}

We close with two natural problems.  The estimate in
Theorem~\ref{thm:orbit-count} does not use the stabilizer of $W$;
determining its orbits on the $s$-subspaces of $F^{2\ell}/W$ would
give a more precise count.  In higher order, it remains open to
exclude arbitrary dimension splits, rather than only the order-one
direct-sum class considered here; exact higher generalized rank
weights may provide the required distinction.
\section*{Acknowledgements}
The author thanks the Italian National Group for Algebraic and Geometric Structures and their Applications (GNSAGA—INdAM)
which supported the research. 

\section*{Declarations}
{\bf Conflicts of interest.} The authors have no conflicts of interest to declare that are relevant to the content of this
article.
\footnotesize
\bibliographystyle{plain}   
\bibliography{ref} 

@article{AMSZ,
  author  = {S. Adriaensen and J. Mannaert and P. Santonastaso and F. Zullo},
  title   = {Cones from maximum $h$-scattered linear sets and a stability result for cylinders from hyperovals},
  journal = {Discrete Mathematics},
  volume  = {346},
  year    = {2023},
  pages   = {113602},
}

@article{BGM,
  author  = {D. Bartoli and A. Giannoni and G. Marino},
  title   = {New scattered subspaces in higher dimensions},
  journal = {Annali di Matematica Pura ed Applicata},
  volume  = {204},
  year    = {2025},
  pages   = {815--834},
}

@article{BMNV,
  author  = {D. Bartoli and G. Marino and A. Neri and L. Vicino},
  title   = {Exceptional scattered sequences},
  journal = {Algebraic Combinatorics},
  volume  = {7},
  year    = {2024},
  pages   = {1405--1431},
}

@article{CMPZ,
  author  = {B. Csajb{\'o}k and G. Marino and O. Polverino and F. Zullo},
  title   = {Generalising the scattered property of subspaces},
  journal = {Combinatorica},
  volume  = {41},
  year    = {2021},
  pages   = {237--262},
}

@article{DGLR,
  author  = {J. de la Cruz and E. Gorla and H. H. L{\'o}pez and A. Ravagnani},
  title   = {Weight distribution of rank-metric codes},
  journal = {Designs, Codes and Cryptography},
  volume  = {86},
  year    = {2018},
  pages   = {1--16},
}

@article{Delsarte,
  author  = {P. Delsarte},
  title   = {Bilinear forms over a finite field, with applications to coding theory},
  journal = {Journal of Combinatorial Theory, Series A},
  volume  = {25},
  year    = {1978},
  pages   = {226--241},
}

@incollection{Ducoat,
  author    = {J. Ducoat},
  title     = {Generalized rank weights: a duality statement},
  editor    = {G. Kyureghyan and G. L. Mullen and A. Pott},
  booktitle = {Topics in Finite Fields},
  series    = {Contemporary Mathematics},
  volume    = {632},
  publisher = {American Mathematical Society},
  year      = {2015},
  pages     = {101--109},
}

@article{Gabidulin,
  author  = {E. M. Gabidulin},
  title   = {Theory of codes with maximum rank distance},
  journal = {Problemy Peredachi Informatsii},
  volume  = {21},
  year    = {1985},
  pages   = {3--16},
}

@article{KMU,
  author  = {J. Kurihara and R. Matsumoto and T. Uyematsu},
  title   = {Relative generalized rank weight of linear codes and its applications to network coding},
  journal = {IEEE Transactions on Information Theory},
  volume  = {61},
  year    = {2015},
  pages   = {3912--3936},
}

@article{MNT,
  author  = {G. Marino and A. Neri and R. Trombetti},
  title   = {Evasive subspaces, generalized rank weights and near {MRD} codes},
  journal = {Discrete Mathematics},
  volume  = {346},
  year    = {2023},
  pages   = {113605},
}

@article{Santonastaso,
  author  = {P. Santonastaso},
  title   = {Completely decomposable rank-metric codes},
  journal = {Linear Algebra and its Applications},
  volume  = {726},
  year    = {2025},
  pages   = {371--404},
}

@misc{BBMS,
  author        = {D. Bartoli and M. Borello and G. Marino and M. Scotti},
  title         = {Linear rank-metric intersecting codes},
  year          = {2025},
  eprint        = {2507.00569},
  archivePrefix = {arXiv},
  note          = {arXiv:2507.00569},
}

@misc{BGMN,
  author        = {D. Bartoli and A. Giannoni and G. Marino and A. Neri},
  title         = {An infinite family of non-extendable {MRD} codes},
  year          = {2026},
  eprint        = {2603.27748},
  archivePrefix = {arXiv},
  note          = {arXiv:2603.27748},
}

@misc{BPZ,
  author        = {M. Borello and O. Polverino and F. Zullo},
  title         = {Delsarte duality on subspaces and applications to rank-metric codes and $q$-matroids},
  year          = {2025},
  eprint        = {2509.24409},
  archivePrefix = {arXiv},
  note          = {arXiv:2509.24409},
}

\end{document}